\documentclass[a4paper,10pt]{amsart}
\usepackage{amsmath,amssymb,amsthm,amscd}
\usepackage[all]{xy}

\newcommand{\Z}{\mathbb{Z}}

\newcommand{\set}[2]{\{#1;#2\}}
\newcommand{\lazar}{\ \&\ }
\newcommand{\N}{\mathbb{N}}
\newcommand{\divchain}[5]{\left(\begin{array}{c|#3}#1 & #4\\#2 & #5\end{array}\right)}
\newcommand{\divchainab}[3]{\divchain{a}{b}{#1}{#2}{#3}}
\newcommand{\lc}{\mathrm{lc}}
\newcommand{\udash}{\underline{\ \ \!}\,}

\usepackage[usenames]{color}
\newcommand{\chge}[1]{\textcolor{red}{#1}}
\renewcommand{\chge}[1]{#1}

\theoremstyle{plain}
\newtheorem{thm}{Theorem}[section]
\newtheorem{prop}[thm]{Proposition}
\newtheorem{lem}[thm]{Lemma}
\newtheorem{cor}[thm]{Corollary}
\theoremstyle{definition}

\theoremstyle{remark}
\newtheorem*{rem}{Remark}

\begin{document}
\title{Quasi-Euclidean subrings of $\mathbb Q[x]$}

\author{\textsc{Petr~Glivick\' y}}
\address{\chge{\textsc{Petr~Glivick\' y}:} Charles University, Faculty of Mathematics and Physics, Department of Theoretical Computer Science and Mathematical Logic \\ 
Malostransk\'e n\'am\v est\'\i\ 12, 118 00 Praha~1, Czech Republic}
\email{petrglivicky@gmail.com}

\author{\textsc{Jan~\v Saroch}}
\address{\chge{\textsc{Jan~\v Saroch}:} Charles University, Faculty of Mathematics and Physics, Department of Algebra \\ 
Sokolovsk\'{a} 83, 186 75 Praha~8, Czech Republic}
\email{saroch@karlin.mff.cuni.cz}


\thanks{First author supported by the grant \chge{GAUK 4372/2011}}
\thanks{Second author supported by the grant E\v CC 301-29/248001.}

\date{\today}

\begin{abstract} Using a nonstandard model of Peano arithmetic, we show that there are quasi-Euclidean subrings of $\mathbb Q[x]$ which are not $k$-stage Euclidean for any norm and positive integer $k$. These subrings can be either PID or non-UFD, depending on the choice of parameters in our construction. In both cases, there are $2^\omega$ such domains up to ring isomorphism.
\end{abstract}

\maketitle
\vspace{4ex}


Although Euclidean and principal ideal domains have been intensively studied for almost a century, examples of non-Euclidean PIDs are still rather scattered throughout the literature, and thought of as more or less singular, non-frequent objects. The oldest of these examples are arguably the rings of integers of $\mathbb Q(\sqrt d)$ for $d = -19, -43, -67, -163$. However, these are the only cases for negative $d$'s, and the results from \cite{W} and \cite{H} indicate that it is almost surely the case of positive values of $d$, too.

Another type of examples was given by Samuel in his famous paper \cite{S}. Leutbecher (in \cite{L}) capitalized on his approach several years later, and proved that there are non-Euclidean PIDs which are even quasi-Euclidean (this was not the case of the four rings of integers mentioned above, as Cohn observed in \cite{Coh}).

Throughout this paper, by a \emph{quasi-Euclidean domain}, we mean a commutative domain $R$ for which there is a function $\phi :R^2\to\omega$ such that, for all $(a,b)\in R^2$ with $b\neq 0$, there exists $q\in R$ with $\phi (b, a-bq) < \phi(a,b)$. The definition is similar to the one of classical Euclidean norm, with the important difference that by the norm function here, we do not measure elements of the ring but pairs of those. Also, unlike in the case of Euclidean domains, $\omega$ can be equivalently replaced by some/any infinite ordinal in the definition; see Preliminaries section (in particular Proposition~\ref{prop:equiv_def}) for this and further equivalent definitions of quasi-Euclidean domain, and related concepts.

There are a few more published results on non-Euclidean PIDs. Unfortunately, they do not usually present a coherent class of these domains, or some sort of characterization of rings which are non-Euclidean PIDs in some distinguished class of domains. Nice attempts in this direction can be found in \cite{A} and \cite{EH}.

In this paper, we present a parametric construction which is in some sense a generalization of the approach used in \cite{EH}. While studying certain models of Peano arithmetic, we noticed that there are many discretely ordered non-Euclidean (even non-$k$-stage Euclidean in the sense of Cooke \cite{Coo}) subrings of $\mathbb Q[x]$ which are quasi-Euclidean. In fact, for each $\tau\in\prod _{p\in\mathbb P}\mathbb J_p$, where $\mathbb J_p$ denotes the ring of $p$-adic integers, we define one such subring. Moreover, we observe that the set $\prod _{p\in\mathbb P}\mathbb J_p$ splits into two parts of full cardinalities, depending on whether the resulting ring is PID or non-UFD. Since each quasi-Euclidean ring is B\' ezout (Proposition~\ref{prop:equiv_def}), there are no inbetween cases, i.e.\ non-PID and UFD at the same time.


\subsection*{Acknowledgements}

The authors would like to thank Josef Ml\v cek and Jan Trlifaj for reading parts of this text and giving several valuable comments.


\section{Preliminaries}
\label{sec:prelim}

Throughout this paper, all rings are (commutative integral) domains. Further, we denote by $\mathbb P$ the set of all primes in $\mathbb N$. For each $p\in\mathbb P$, $\mathbb J_p$ stands for the ring of $p$-adic integers, while $\mathbb Z_p$ denotes the field $\mathbb Z/p\mathbb Z$. Since $\mathbb J_p\cong\varprojlim\mathbb Z_{p^k}$, we shall view $\mathbb J_p$ as a subring of $\prod _{k=1}^\infty\mathbb Z_{p^k}$, and denote, for a positive integer $k$, by $\pi _k$ the canonical projection from $\mathbb J_p$ to $\mathbb Z_{p^k}$. It will not cause any confusion that the notation $\pi _k$ does not reflect the prime $p$. Moreover, for technical reasons, we put $\pi_0: \mathbb J_p \to \{0\}$; again, regardless of the prime $p$.

If we deal with elements from the ring $\mathbb Q[x]$, we define $\deg 0 = -1$, and we denote by lc$(q)$ the leading coefficient of a polynomial $q$.


\subsection{Quasi-Euclidean and $k$-stage Euclidean domains}

Various generalizations of the concept of a Euclidean domain were proposed and studied in the past. The one we find very natural, is the concept of quasi-Euclidean (used in \cite{L} and \cite{B}) or the equivalent notion of $\omega$-stage Euclidean domain \chge{(used by Cooke in \cite{Coo})}.

Given a ring $R$ and a partial order $\leq$ on $R^2$, we say that $\leq$ is \emph{quasi-Euclidean} if it has the descending chain condition (dcc), and for any pair $(a, b)\in R^2$ with $b\neq 0$, there exists an element $q$ in $R$ such that $(b,a-bq)<(a,b)$. We call $R$ \emph{quasi-Euclidean} provided there exists a quasi-Euclidean partial order on $R^2$.

Let $(a, b)\in R^2$ and $k$ be a \chge{non-negative} integer. A \emph{$k$-stage division chain} starting from the pair $(a, b)$ is a sequence of equations in $R$
\begin{align*} a &= q_1b + r_1 \\
b &= q_2r_1 + r_2 \\
r_1 &= q_3r_2 + r_3 \\
&\qquad\vdots\\
r_{k-2} &= q_kr_{k-1} + r_k.
\end{align*}
Such a division chain is called \emph{terminating} if the last remainder $r_k$ is $0$ ($r_{k-1}$ is then easily seen to be the GCD of $a$ and $b$). Notice that a $k$-stage division chain is determined by its starting pair and the sequence of quotients $q_1, \dotsc , q_k$. For the sake of compactness, in what follows, we shall denote this chain also by 
\begin{equation*}\label{divchaindef}\divchainab{ccc}{q_1 & \ldots & q_k}{r_1 & \dots & r_k}.\end{equation*}
\noindent Given such a division chain, we define its $0$-th remainder $r_0$ as $b$.

In the following proposition, $On$ denotes the class of all ordinal numbers.

\begin{prop} \label{prop:equiv_def} {\rm (\cite{B}, \cite{Coo}, \cite{L})} For a commutative domain $R$, the following conditions are equivalent:
\begin{enumerate}
\item There exists a function $\phi :R^2\to On$ (with Rng$(\phi)\subseteq\omega$) such that, for all $(a,b)\in R^2$ with $b\neq 0$, there exists $q\in R$ such that $\phi (b, a-bq) < \phi(a,b)$.
\item $R$ is quasi-Euclidean.
\item $R$ is a B\' ezout domain, and the group $\hbox{\rm GL}_2(R)$ of regular $2\times 2$ matrices over $R$ is generated by matrices of elementary transformations.
\item Every pair $(a, b)\in R^2$ with $b\neq 0$ has a terminating $k$-stage division chain for some positive integer $k$.
\end{enumerate}
\end{prop}

\begin{proof} $(1) \Longrightarrow (2)$ is trivial, we just put $(a, b)<(a^\prime, b^\prime)$ if $\phi(a, b)< \phi(a^\prime, b^\prime)$.

$(2) \Longrightarrow (4)$ follows directly by the dcc.

The equivalence of $(3)$ and $(4)$ was proved already in \cite[14.3]{O}.

$(4)\Longrightarrow (1)$: We put $\phi (a,0) = 0$ for all $a\in R$. If $b \neq 0$, we define $\phi (a, b)$ as the minimal $k\in\omega$ for which the pair $(a, b)$ has a terminating $k$-stage division chain. (So we even manage to find $\phi$ with the range in $\omega$.)
\end{proof}

Notice that no notion of a norm is involved in the definition of a quasi-Euclidean domain. However, given a norm $N$ on $R$ (i.e. a function $N: R \to\mathbb N$ with $N(a) = 0$ iff $a = 0$), we can measure how far $N$ is from being Euclidean: as in \cite{Coo}, \chge{for $0<k\leq \omega$}, we say that $R$ is a \emph{$k$-stage Euclidean domain with respect to $N$} provided that, for every $(a,b)\in R^2$ with $b\neq 0$, there 
\chge{exists a positive integer $l\leq k$ such that for some $l$-stage division chain starting from $(a,b)$ it is $N(r_l)<N(b)$.}
As usual, we say that $R$ is \emph{$k$-stage Euclidean} if there exists such a norm $N$ on $R$. So, in our notation, $1$-stage Euclidean means Euclidean (in the classic sense). \chge{On the other hand, by Proposition \ref{prop:equiv_def}, $R$ is $\omega$-stage Euclidean (with respect to some/any norm) if and only if it is quasi-Euclidean.}

Finally, observe that a quasi-Euclidean domain, being B\' ezout, is UFD if and only if it is PID. An example of non-UFD $2$-stage Euclidean domain was given already by Cooke in \cite{Coo}, at the end of \S 1. It is at this place, where he admits that he does not know of any example of quasi-Euclidean domain which is not $2$-stage Euclidean. Interestingly, all examples, we are going to construct, have got this property.


\subsection{Peano arithmetic and weak saturation}

Although our construction will be purely algebraic, we are going to give also a description derived from a nonstandard model of Peano arithmetic (PA). There are several reasons to do this: the description is very natural, only basic logical tools are needed, and it sheds more light at the entire situation.

Our models of PA are thought of as models in the language $L = (0,1,+,\cdot, \leq)$. The fact that it is an extension of the language of rings will make it more convenient for us to work with. In particular, we can immediately say that any model of PA is a (discretely ordered) commutative semiring with $0$ and $1$.

We will say that $\mathcal M\models$ PA is \emph{weakly saturated} if every $1$-type in $\mathcal M$ without parameters is realized in $\mathcal M$, i.e. given any set $Y = \{\varphi_i (x) \mid i\in I\}$ of $L$-formulas with one free variable $x$, there is $m\in M$ such that $\mathcal M\models\varphi_i[m]$ for all $i\in I$, provided that, for each finite subset $S$ of $I$, one has $\mathcal M\models(\exists x)\bigwedge _{i\in S}\varphi _i(x)$. Indeed, weakly saturated models of PA exist, we can even take an appropriate elementary extension of $\mathbb N$, however, as we shall see, for such a model $\mathcal M$, necessarily $|M|\geq 2^\omega$.


\section{Examples}
\label{sec:ex}


\subsection{Logical description}

Let us fix a weakly saturated model $\mathcal M$. Then, as mentioned above, $\mathcal M$ forms a commutative semiring. Formally adding negative elements, we turn $\mathcal M$ into a commutative domain containing $\mathbb Z$ as a subring. We will denote this domain $\mathcal M^\pm$. Notice that $\mathcal M^\pm$ shares several basic properties with $\mathbb Z$, namely it is a discretely ordered GCD domain; also for every $q,r$ with $r\neq 0$, there exists $0\leq t<|r|$ such that $r$ divides $q+t$ (where \chge{$|\udash|$} is the usual absolute value). However, unlike $\mathbb Z$, $\mathcal M^\pm$ is not Noetherian.

Let $a$ be a nonstandard element of $\mathcal M$, i.e. $a\in M\setminus\mathbb N$. We define a subring $R_a$ of $\mathcal M^\pm$ in the following way:

$$R_a = \{m\in M^\pm \mid (\exists n\in\mathbb N)(\exists h\in\mathbb Z[x])\, n\neq 0 \;\&\; n\cdot m = h(a)\}.$$

It is easily seen that $R_a$ is a ring. It can be naturally approached if we, in the first step, take a subring of $\mathcal M^\pm$ generated by $a$ (which is nothing else than $\mathbb Z[a]\cong\mathbb Z[x]$), and then allow division by nonzero integers in case it is possible in $\mathcal M^\pm$. We immediately observe that $R_a$ is isomorphic to

$$R_a^\prime = \left\{\frac{h}{n}\in\mathbb Q[x]\;\Bigl|\Bigr.\; n\in\mathbb N\setminus\{0\}, h\in\mathbb Z[x]\hbox{, and }n\,|\,h(a)\hbox{ in }\mathcal M^\pm\right\}.$$

\begin{rem} \hfill\par
\begin{enumerate}
\item Regardless of $a$, we have $R_a^\prime \cap \mathbb Q = \mathbb Z$.
\item Notice that $R_a = R_{a+1}$ (for any nonstandard $a\in M$) but $R_a^\prime \neq R_{a+1}^\prime$ since precisely one of these two rings contains $x/2$. On the other hand, as we shall see later, it is possible that we have nonstandard $a,b\in M$ such that $R_a \neq R_b$ but $R_a^\prime = R_b^\prime$.
\item For our considerations, we do not need the full strength of PA. In fact, instead of binary multiplication, it is enough to have an endomorphism $a\cdot$ of the monoid $(M,+,0)$ such that $a\cdot 1\not\in\mathbb N$, and the induction for all formulas in the language $(0,1,+,a\cdot,\leq)$; so the resulting theory can be viewed as an extension of Presburger arithmetic rather than weakening of PA. \chge{In fact, Theorem \ref{thm:quasi} was obtained 
as a part of the first author's 
proof of model-completeness of this theory.}
\end{enumerate}
\end{rem}


\subsection{Algebraic description}

As we have seen above, the definitions of $R_a$ and $R_a^\prime$ rely on the fixed model $\mathcal M$ of PA. However, there is only a little amount of information about $a\in M$ that we actually need. This makes it possible---as we are going to demonstrate---to manage without refering to any Peano model. For $\tau\in\prod_{p\in\mathbb P}\mathbb J_p$, we define a subring $R_\tau$ of $\mathbb Q[x]$.

$$R_\tau = \left\{\frac{h}{n}\in\mathbb Q[x]\;\Bigl|\Bigr.\; n\in\mathbb N\setminus\{0\}, h\in\mathbb Z[x]\hbox{, and }(\forall p\in\mathbb P)\,\pi _{\hbox{v}_p(n)}(h(\tau _p)) = 0\right\}.$$

Here, $\hbox{v}_p$ denotes the usual $p$-valuation. Further, $\tau _p$ is the $p$th projection of $\tau$, and the substitution $h(\tau _p)$ is done inside $\mathbb J_p$ where $\mathbb Z$ is canonically embedded via $z\mapsto (z\!\!\mod p, z\!\!\mod p^2, z\!\!\mod p^3,\dotsc)$. We will use this substitution several times in the next section.

It follows easily from the definition that $\sigma \neq \tau$ implies $R_\sigma \neq R_\tau$. The correspondence between the rings $R_a^\prime$ and $R_\tau$ is made precise by Proposition~\ref{prop:corresp}.

\begin{prop} \label{prop:corresp} Let $\mathcal M$ be a weakly saturated model of PA. Then:
\begin{enumerate}
\item For each nonstandard $a\in M$ there exists precisely one $\tau\in\prod_{p\in\mathbb P}\mathbb J_p$ such that $R_a^\prime = R_\tau$.
\item For each $\tau\in\prod_{p\in\mathbb P}\mathbb J_p$ there is at least one nonstandard $a\in M$ such that $R_a^\prime = R_\tau$.
\end{enumerate}
\end{prop}

\begin{proof} $(1)$ There is even a ring homomorphism $\psi: \mathcal M^\pm \to \prod_{p\in\mathbb P}\mathbb J_p$ which assigns to $m\in M^\pm$ an element $\tau$ such that $\tau _p = (m\!\!\mod p,m\!\!\mod p^2,m\!\!\mod p^3,\dotsc)$ for each $p\in\mathbb P$. It is a matter of straightforward verification that $R_a^\prime = R_{\psi(a)}$ for any nonstandard $a\in M$.

$(2)$ Let us consider the set $Y$ consisting of all congruences $x\equiv _{p^k}\tau _p(k)$ and inequalities $x > k$, where $k\in\mathbb N\setminus\{0\}$ and $p\in\mathbb P$. Then $Y$ is a $1$-type in $\mathcal M$ (without parameters---positive integers are just constant terms in the language $L$) since any finite subset of $Y$ has a solution in $\mathbb N\subset M$ by Chinese Remainder Theorem. So there is a global solution, $a\in M$, of all congruences and inequalities from $Y$, using the weak saturation of $\mathcal M$. (Now, it is clear that $|M|\geq 2^\omega$.) The inequalities assure that $a$ is nonstandard, and checking the definitions, we immediately see that $R_a^\prime = R_\tau$.
\end{proof}

In the following section, we will freely use the fact (implicitly proved above) that, for every $\tau$, the ring $R_\tau$ inherits the discrete ordering from $\mathcal M^\pm$ via isomorphism with $R_a$ for some/any $a$.


\section{Properties of the examples}
\label{sec:prop_ex}


\subsection{Terminating division chains}

We are going to show that, for every $\tau$, the ring $R_\tau$ is quasi-Euclidean. So let $\tau$ be fixed for a while, put $R = R_\tau$, and let us denote by $R^+$ the subsemiring of $R$ consisting of polynomials with nonnegative leading coefficients. First, we prove the following auxiliary result.

\begin{lem} \label{lem:aux} Let 
$q,r\in R^+$ with $r\neq 0$, then there are (unique) $p,s\in R^+$ such that $q = pr + s$ and $s<r$.

\chge{Moreover: Let} $\tilde p, \tilde s\in\mathbb Q[x]$ be such that $q = \tilde pr+\tilde s$ and $\deg \tilde s < \deg r$. Further let $\tilde p = p^\prime /m$ where $p^\prime \in\mathbb Z[x]$, $m\in\mathbb N\setminus\{0\}$ and $0\leq k < m$ such that $(p^\prime - k)/m\in R^+$. Then the pair $(p,s)$ satisfies
$$(p, s) = \left\{\begin{array}{ll}
\left(\tilde p - 1,\tilde s+r\right) & \mbox{\rm for } k = 0\;\&\;\mbox{\rm lc}(\tilde s) <0, \\
\left(\frac{p^\prime - k}{m}, \tilde s + \frac{k}{m}r\right) & \mbox{\rm otherwise}.\end{array}\right.$$
\end{lem}

\begin{proof} 
Straightforward verification.
\end{proof}

If we look at $R_a$ (for $a$ with $R_a^\prime = R$), there is only one pair $(p, s)$ in the model $\mathcal M$ satisfying the properties from Lemma~\ref{lem:aux}, namely the pair $(q \ \mathrm{div}\ r, \mbox{$q$ mod $r$})$. Here, div stands for the binary operation of integer division. Thus in particular, we have that $R^+$ as a subsemiring of $\mathcal M$ is closed under binary operations div and mod.

Consequently, we say that a division chain $\divchain{r_{-1}}{r_0}{ccc}{q_1 & \ldots & q_n}{r_1 & \ldots & r_n}$ in $R^+$ with $r_{-1},r_0>0$ is \emph{quasi-Euclidean} if $q_{i+1} = r_{i-1}\mbox{ div } r_i$ and $r_{i+1} = r_{i-1}\mbox{ mod } r_i$, for $i\geq 0$. A consequence of the proof of the following theorem is that, for any nonzero $a,b\in R^+$, there exists a positive integer $n$ such that the quasi-Euclidean chain of length $n$ starting from the pair $(a,b)$ is terminating.

\begin{thm} \label{thm:quasi} $R$ is a quasi-Euclidean domain. In particular, it is B\' ezout.
\end{thm}

\begin{proof} We will show that the condition $(1)$ from Proposition~\ref{prop:equiv_def} is satisfied. For this sake, we define $\phi: R^2 \to (2\times\mathbb N^4, lex)$ by the formula $\phi(q,r) = (0,0,0,0,0)$ for $r = 0$, and
$$\phi(q,r) = (\delta _{q,r}, \deg q + 1, \deg r, n_{q,r},n_{q,r}\cdot |\mbox{lc}(q)|)$$ otherwise. Here, $\delta _{q,r}$ is $1$ if $|q|\leq |r|$, and $0$ otherwise; $n_{q,r}\in\mathbb N$ denotes the least common denominator of $q,r$. In the rest of the proof, we assume that $q>r>0$. The other cases follow easily. (Notice that $\phi(q,r) = \phi(|q|,|r|)$.)

Since $\mathbb Q[x]$ is a Euclidean ring with the norm $\deg(-) + 1$, there are $\tilde p, \tilde s\in\mathbb Q[x]$ such that $q = \tilde pr+\tilde s$ and $\deg \tilde s < \deg r$. By Lemma~\ref{lem:aux}, we get $p,s\in R^+$ satisfying $s<r$ and $q = pr + s$.

Suppose $s\neq 0$. We need to show that $\phi (r,s)<\phi(q,r)$ in the lexicographic order of $2\times\mathbb N^4$. Since $0<s<r$, we have $\delta _{r, s} = 0 = \delta _{q,r}$. We may assume $\deg q = \deg r = \deg s$ (otherwise, we are done immediately). Then $p\in\mathbb N$. Further, we have $q,r\in\frac{\mathbb Z[x]}{n_{q,r}}$, and hence $s = q - pr\in\frac{\mathbb Z[x]}{n_{q,r}}$. Therefore $n_{r,s}\leq n_{q,r}$. Moreover, from $r<q$, we have $\mbox{lc}(r)\leq\mbox{lc}(q)$.

Assume $n_{r,s} = n_{q,r}$ and $\mbox{lc}(r) = \mbox{lc}(q)$. Then, from the definition of $\tilde p$, we have $\tilde p = 1$, and thus $p^\prime = 1 = m$, $k = 0$ in Lemma~\ref{lem:aux}. The first case in the definition of $(p,s)$ leads to a contradiction, since we get $p = 0$ (and so $q=s<r$). So it must be that $p = \tilde p = 1$ and $s = \tilde s$. In particular, we see that $\deg s = \deg \tilde s < \deg r$ which also contradicts one of our assumptions.

Finally, $R$ is B\' ezout by Proposition~\ref{prop:equiv_def}.
\end{proof}


\subsection{Separating the PID cases}

In the following few paragraphs, we distinguish the choices of $\tau$ which imply that $R_\tau$ is a PID. We also show that there are $2^\omega$ pairwise nonisomorphic domains among the rings $R_\tau$ which are PID, and the same cardinality of those which are not PID. The next lemma will be useful.

\begin{lem} \label{lem:PID} Let $\tau\in\prod _{p\in\mathbb P}\mathbb J_p$. Then $R_\tau$ is a PID if and only if, for each nonzero $h\in\mathbb Z[x]$, the set $S_h = \{(p,k)\in\mathbb P\times (\mathbb N\setminus\{0\})\mid \pi _k (h(\tau _p)) = 0\}$ is finite.
\end{lem}

\begin{proof} Assume that $S_h$ is infinite for some nonzero $h\in\mathbb Z[x]$. Then either the set $\{p\in\mathbb P\mid h/p\in R_\tau\}$ is infinite, or there exists a prime $p$ such that $h/p^k\in R_\tau$ for any $k\in\mathbb N$. In the first case, we fix an enumeration $\{p_1, p_2, p_3,\dotsc\}$ of that set, and---using the definition of $R_\tau$---we see that $(h/p_1, h/(p_1p_2), h/(p_1p_2p_3),\dotsc)$ is an infinite descending (with respect to divisibility) sequence of elements in $R_\tau$; thus $R_\tau$ is not a UFD. In the second case, we use the same argument for the sequence $(h/p, h/p^2,h/p^3,\dotsc)$.

If $R_\tau$ is not a PID, then (since it is B\' ezout by Theorem~\ref{thm:quasi}) there has to be an infinite sequence of elements in $R_\tau$ descending in divisibility $(h_1/n_1, h_2/n_2,\dotsc)$; here $h_i\in\mathbb Z[x]$ and $n_i$ are positive integers coprime with $h_i$ in $\mathbb Z[x]$, for all $i>0$. The polynomials $h_i$ will eventually have the same degree ($\mathbb Q[x]$ is Euclidean) and absolute value of the leading coefficient ($\mathbb Z$ is Noetherian), and so we may w.l.o.g.\ assume that all the polynomials $h_i$ are equal to a single nonzero $h\in\mathbb Z[x]$. It directly follows that, for this $h$, the set $S_h$ is infinite.
\end{proof}

Let us take a representative subset $J$ of $\prod _{p\in\mathbb P}\mathbb J_p$ in the sense that, for each $\rho$, there is a $\tau\in J$ such that $R_\tau\cong R_\rho$, and for all $\tau, \sigma\in J$, $\tau\neq\sigma$, we have $R_\tau\not\cong R_\sigma$. Then $J$ is a disjoint union of the sets $A$ and $B$, where $A = \{\tau\in J \mid R_\tau\mbox{ is a PID}\}$ and $B = \{\tau\in J \mid R_\tau\mbox{ is not a UFD}\}$.

\begin{prop} $|A| = |B| = 2^\omega$.
\end{prop}

\begin{proof} Let us assume that $|A|< 2^\omega$. For each $p\in\mathbb P$, we define $\tau _p\in\mathbb J_p$ in such a way that:
\begin{enumerate}
\item $\pi _1(\tau _p) = \lfloor \log p \rfloor$,
\item $n\cdot\tau _p\not\in\{h(\sigma _p) \mid \sigma\in A \;\&\; h\in\mathbb Z[x]\},\hbox{ for every positive integer }n$,
\item $\tau _p$ is not a root in $\mathbb J_p$ of a nonzero polynomial from $\mathbb Z[x]$.
\end{enumerate}

This is clearly possible since the first two conditions are satisfied by $2^\omega$ different elements of $\mathbb J_p$. Let $\tau = \prod _{p\in\mathbb P}\tau _p$. We claim that $R_\tau$ is a PID which leads immediately to a contradiction (by $(2)$, there cannot be $\sigma\in A$ with $R_\tau\cong R_\sigma$).

To prove this, we use Lemma~\ref{lem:PID}. Let us fix a nonzero $h\in\mathbb Z[x]$. Then, using the limit comparison of $h$ and $\log$, we deduce that, for all sufficiently large primes $p$, we have $0 < |h(\lfloor \log p \rfloor)|<p$ which further implies $\pi _1(h(\tau _p)) \neq 0$. Together with the condition $(3)$, we get that $S_h$ is finite. This finishes the proof that $|A| = 2^\omega$.

To see that $|B| = 2^\omega$, it is enough to fix a $\sigma\in A$, and for each nonzero subset $P$ of $\mathbb P$ define $\tau ^P\in B$ by setting $\tau _p^P = (0,0,0,\dotsc)$ for $p\in P$, and $\tau _p^P = \sigma _p$ otherwise.
\end{proof}



\subsection{Keeping distance from Euclidean domains}

Here, we prove that no $R_\tau$ is a $k$-stage Euclidean domain, whatever positive integer $k$ we take. From now on, we work in a fixed ring $R_\tau$. We start with two slightly technical lemmas \footnote{Lemma \ref{lem:jedna} is a modified version of a classical result on continued fractions by Perron \chge{(see \cite{P}).}}.

\begin{lem} \label{lem:jedna}
Let $Q = \divchainab{ccc}{q_1 & \ldots & q_k}{r_1 & \ldots & r_k}$ be a division chain starting from $(a,b)$ with $a,b,k>0$. There is a division chain $Q' = \divchainab{ccc}{q'_1 & \ldots & q'_l}{r'_1 & \ldots & r'_l}$ with $q'_i>0$ for $i>1$ such that $|r_k| = |r'_l|$ and $l\leq 2k-1$.
\end{lem}

\begin{proof}
Denote $T_1, T_2$ the following two transformations on the set of all division chains starting from $(a,b)$:
$$
T_1: \divchainab{ccc}{q_1 & \ldots & q_k}{r_1 & \ldots & r_k} \mapsto\hfill$$
$$\divchainab{ccccccccc}{q_1 & \ldots & q_{i-1} & q_i-1 & 1 & -(q_{i+1}+1) & -q_{i+2} & \ldots & -q_k}{r_1 & \ldots & r_{i-1} & r_i+r_{i-1} & -r_i & (-1)^2r_{i+1} & (-1)^3r_{i+2} & \ldots & \pm r_k}
$$
\smallskip
where $i$ is the first index such that $q_{i+1}<0$ ($T_1$ is identity if there is no such $i$) and $\pm$ stands for $(-1)^{k-i+1}$;

$$
T_2: \divchainab{ccc}{q_1 & \ldots & q_k}{r_1 & \ldots & r_k} \mapsto \divchainab{ccccccc}{q_1 & \ldots & q_{i-1} & q_i + q_{i+2} & q_{i+3} & \ldots & q_k}{r_1 & \ldots & r_{i-1} & r_{i+2} & r_{i+3} & \ldots & r_k}
$$
where $i$ is the first index such that $q_{i+1}=0$ ($T_2$ is identity if there is no such $i$).

We will show a little bit more than stated---instead of $l\leq 2k-1$, we prove even that $l\leq k + n$ where $n=\max\set{k-i+1}{i>1 \lazar q_i<0}$ ($n=0$ if there is no such $i$). Put $Q = \divchainab{ccc}{q_1 & \ldots & q_k}{r_1 & \ldots & r_k}$ and denote the corresponding pair $(n,k)$ as $p_Q=(n_Q,k_Q)$. We prove the statement by induction on the pairs $(n_Q,k_Q)$ with lexicographic ordering. The case $p_Q = (0,1)$ is trivial.

If there is $i$ such that $q_{i+1}=0$, we get $p_{T_2(Q)}\leq_{lex}(n_Q,k_Q-2)$, and the induction assumption gives some $Q'$. It is easy to verify that this $Q'$ meets all the requirements. \chge{(Note that in the case $i+1=k$ we get $T_2(Q)=\divchainab{ccc}{q_1 & \ldots & q_{i-1}}{r_1 & \ldots & r_{i-1}}$ and $r_{i-1}=r_{i+1}$.)}

Otherwise we have $q_i \neq 0$ whenever $i>1$, and using $T_1$ we get $p_{T_1(Q)}\leq_{lex} (n_Q-1,k_Q+1)$. Again, the $Q'$ given by the induction assumption is what we wanted.
\end{proof}

\begin{lem} \label{lem:dve}
Let $\divchainab{ccc}{q_1 & \ldots & q_k}{r_1 & \ldots & r_k}$ be a division chain starting from $(a,b)$ such that $a,b,q_i>0$ for $i>1$, and let $\divchainab{ccc}{e_1 & \ldots & e_m}{f_1 & \ldots & 0}$ be the quasi-Euclidean division chain in $R_\tau$ starting from $(a,b)$. Assume $m\geq k$. 

Then $|r_k|\geq f_{k+1}$, and in particular $\deg(r_k)\geq\deg(f_{k+1})$ (we put $f_{k+1}=0$ if $m=k$).
\end{lem}
\begin{proof}
Take the least $l$ such that $q_l\neq e_l$ (if there is no such, we are done since $(f_i)$ is decreasing). By an inductive argument, it is easy to observe that the following holds (\chge{recall that} we put $f_0 = r_0 = b$):
\medskip
\\
If $q_l<e_l$ then 
$\left\{
\begin{array}{l}
r_{l+2i}\geq r_{l-1}\ \mathrm{for}\ i\geq 0, \\
r_{l+2i+1}\leq -r_{l-1}\ \mathrm{for}\ i\geq 1, \\
r_{l+1}\leq -r_{l-1} \ \mathrm{or}\ r_{l+1}=-f_l;
\end{array}\right.
$
\medskip
\\ and if $q_l>e_l$ then 
$\left\{
\begin{array}{l}
r_{l+2i} < - r_{l-1}\ \mathrm{for}\ i\geq 1, \\
r_{l+2i+1} > r_{l-1}\ \mathrm{for}\ i\geq 0, \\
r_{l}\leq -r_{l-1}\ \mathrm{or}\ (m>k \;\&\; r_{l}\leq-f_{l+1}).
\end{array}\right.
$
\medskip

\noindent The statement follows since $r_{l-1}=f_{l-1}$ and $(f_i)$ is decreasing.
\end{proof}



Combining both lemmas together, we obtain the following corollary which gives us a bound on the speed of decrease of remainders in a division chain, \chge{compared to} the quasi-Euclidean one. By letting $a,b$ be any two consecutive Fibonacci numbers, one can see that the bound is optimal.

\begin{cor}
\label{cor:twotogether}
Given $a,b>0$, let $\divchainab{ccc}{e_1 & \ldots & e_n}{f_1 & \ldots & 0}$ be the quasi-Euclidean division chain starting from $(a,b)$, and $\divchainab{ccc}{q_1 & \ldots & q_k}{r_1 & \ldots & r_k}$ be an arbitrary division chain. Then, for $l\leq \min(k,n/2)$, we have $|r_l|\geq f_{2l}$.
\end{cor}

Now, we have all the tools for proving that no $R_\tau$ is $k$-stage Euclidean domain, independently of the choice of $k>0$. For the sake of better readability, we state the key step of the proof as a separate lemma.

\begin{lem}
\label{lem:klengthchain}
Let $k$ be a positive integer and $0<b\in R_\tau$ such that $\deg(b)\geq 1$. Then there is $0<a\in R_\tau$ such that every division chain $\divchainab{ccc}{q_1 & \ldots & \chge{q_l}}{r_1 & \ldots & \chge{r_l}}$ of length \chge{$l\leq k$} starting from $(a,b)$ satisfies $\deg(\chge{r_l})\geq\deg(b)$.
\end{lem}

\begin{proof}
By Corollary~\ref{cor:twotogether}, it is enough to prove the statement for the quasi-Euclidean division chain instead of an arbitrary one.

Set $a=\frac{c}{d}(b-\beta)$ where $c,d\in\N$ are such that no division chain in $\mathbb{Z}$ of length \chge{$l\leq k$} starting from $(c,d)$ is terminating (such $c,d$ exist since Corollary~\ref{cor:twotogether} holds also in $\Z$) and $0\leq\beta<d$ is such that $d|(b-\beta)$ in $R_\tau$.

For a contradiction, let the quasi-Euclidean division chain $$\divchainab{ccc}{e_1 & \ldots & \chge{e_l}}{f_1 & \ldots & \chge{f_l}}$$ starting from $(a,b)$ satisfy $\deg(\chge{f_l})<\deg(b)$. We may w.l.o.g. assume $\deg(\chge{f_{l-1}})=\deg(b)$; then we have $e_i\in\Z$ for all $i = 1,2,\dotsc ,\chge{l}$. 

Define the operation $\hat{\ }\!: R_\tau \rightarrow \mathbb{Q}$ as $\hat{r} = \lc(dr)/\lc(b)$. Easily $\hat{a},\hat{b}\in\mathbb Z$, and therefore also $\hat{f}_i\in\mathbb Z$, for all $i\neq l$. Hence, $$\divchain{\hat{a}}{\hat{b}}{cccc}{e_1 & \ldots & \chge{e_{l-1}} & \chge{e_l}}{\hat{f}_1 & \ldots & \chge{\hat{f}_{l-1}} & 0}$$ is a division chain in $\mathbb Z$ starting from $(\hat{a},\hat{b})=(c,d)$, a contradiction.

\end{proof}

\begin{thm}
Let $\tau\in\prod _{p\in\mathbb P}\mathbb J_p$ be arbitrary. Then the ring $R_\tau$ is not $k$-stage Euclidean for any positive integer $k$.
\end{thm}

\begin{proof}
Assume the contrary and let $N$ be a norm such that $R_\tau$ is $k$-stage Euclidean with respect to $N$. To get a contradiction, we construct an infinite sequence $(b_0,b_1,\ldots)$ of elements from $R_\tau$ with $N(b_i)>N(b_{i+1})$ and $\deg{b_{i+1}\geq \deg{b_i}\geq 1}$, for all $i\in\N$.

As the first step, put $b_0=x\in R_\tau$. Now assume we have defined $b_i$ for all $i\leq j\in\N$. Suppose $b_j>0$. For $b_j$ we find some $a_j$ using Lemma~\ref{lem:klengthchain}. By the $k$-stage Euclidean property, there is an \chge{$l$-stage} division chain $\divchain{a_j}{b_j}{ccc}{q_1 & \ldots & \chge{q_l}}{r_1 & \ldots & \chge{r_l}}$ \chge{with $l\leq k$} starting from the pair $(a_j,b_j)$ such that $N(\chge{r_l})<N(b_j)$. So we can set $b_{j+1}=\chge{r_l}$. By Lemma~\ref{lem:klengthchain}, we know that $\deg b_{j+1}\geq \deg{b_j} \geq 1$. 

The case $b_j<0$ is similar. For $-b_j$ find $-a_j$ by Lemma \ref{lem:klengthchain}, take a division chain $\divchain{a_j}{b_j}{ccc}{q_1 & \ldots & \chge{q_l}}{r_1 & \ldots & \chge{r_l}}$ with $N(\chge{r_l})<N(b_j)$ and set $b_{j+1}=\chge{r_l}$. If $\deg \chge{r_l} < \deg b_j$, we would have the division chain $\divchain{-a_j}{-b_j}{ccc}{q_1 & \ldots & \chge{q_l}}{-r_1 & \ldots & \chge{-r_l}}$ with $\deg \chge{-r_l} < \deg -b_j$, contradicting the choice of $-a_j$.
\end{proof}
\bigskip
We conclude our paper by the following
\subsection*{Open question}

Is there an example of a $k$-stage Euclidean domain which is not $(k-1)$-stage Euclidean, for $k>2$?


\end{document}